\documentclass[11pt]{amsart}
\usepackage{latexsym, amsmath, amssymb,amsthm,amsopn,amsfonts}
\usepackage{version}
\usepackage{epsfig,graphics,color,graphicx,graphpap}
\usepackage{amssymb}
\usepackage{todonotes}
\usepackage{listings}
 
\usepackage[framemethod=tikz]{mdframed}

\usepackage[citecolor=blue]{hyperref}
\usepackage[framemethod=tikz]{mdframed}

\numberwithin{equation}{section}
\newtheorem{theorem}{Theorem}[section]

\newtheorem{proposition}[theorem]{Proposition}
\newtheorem{lemma}[theorem]{Lemma}

\newtheorem{remark}{Remark}[section]

\newcommand{\LV}{\left|}
\newcommand{\RV}{\right|}
\newcommand{\LC}{\left(}
\newcommand{\RC}{\right)}

\newcommand{\p}{\partial}

\newcommand{\R}{\mathbb R}

\title[ ]{Quench detection on a superconducting radio-frequency cavity}

\author[Lai]{Ru-Yu Lai}
\address{School of Mathematics, University of Minnesota, Minneapolis, MN 55455, USA}
\curraddr{}
\email{rylai@umn.edu }
\thanks{}

\author[Spirn]{Daniel Spirn }
\address{School of Mathematics,  University of Minnesota, Minneapolis, MN 55455, USA}
\curraddr{}
\email{spirn@umn.edu}
\thanks{  }

\begin{document}
\maketitle
\begin{abstract}
We study quench detection in superconducting accelerator cavities cooled with He-II. 
A rigorous mathematical formula is derived to localize the quench position from dynamical data over a finite time interval at a second sound detector.

\end{abstract}
\section{Introduction}
Superconducting radio-frequency (SRF) cavities are an important technology for particle accelerators. Unlike  normal conducting cavities,  SRF cavities have the ability to withstand high electromagnetic fields and efficiently accelerate particle beams \cite{Conway}. These cavities, typically made of niobium, are bathed in liquid helium at low temperatures (1.6K-2.0K) to maintain  superconductivity.  
However, the appearance of defects on the surface of the cavity can enhance localized power losses in a cavity, which leads to heating of the surrounding area; and eventually, local superconductivity will return to a normal conducting state after reaching a critical temperature. 
This transition from superconducting to normal conducting, called a \textit{quench}, can happen quickly \cite{Conway}. It severely limits the accelerating gradient of the cavity
and can lead to damage of the device. For quality assurance of the cavity, it is crucial to identify defects by finding any localized quench positions. 


We introduce two quench localization systems, temperature mapping system and second sound system, that are widely used in laboratories. Many other diagnostic techniques exist, such as X-ray radiation mapping \cite{Sakai} and optical inspection, and we refer to \cite{Conway2017} and the references therein for more information. The temperature mapping system (T-mapping) \cite{Conway2017} is a traditional approach for locating a quench position by measuring small temperature increases on the exterior of SRF cavities in liquid helium bath. Thousands of temperature sensors are placed on the outer cavity surface to develop a mapping system that helps pinpoint hot spots in the cavity wall, known as quench spots. This technique is fairly time consuming compared with second sound system, discussed below, which requires less time and fewer detectors. 
The second sound system is based on a special property of superfluid helium in which normal and superfluid components coexists. This is known as the two fluid model of Tisza, and it was later refined by Landau \cite{Landau, Tisza}.

In the two fluid model, superfluid helium is described as an interpenetrating mixture of a normal component, which behaves like normal fluid, and a superfluid component with zero viscosity and zero entropy. It consists of conservation laws for the total density $\rho=\rho_s+\rho_n$, entropy $S$ per unit mass, and the total momentum $j=\rho_sv_s+\rho_nv_n$, where subscript $s$ and $n$ denote, respectively, the superfluid and normal components. Here $v_s$ and $v_n$ denote the superfluid and normal velocities. Let $P$ be the pressure and $\mathcal{T}$ be the temperature.
At low temperatures, since the entropy depends significantly on temperature, it leads to the following second sound wave:
$$
q \p_t^2 \delta \mathcal{T}- \Delta \delta \mathcal{T}=0,
$$
where $\mathcal{T}=\mathcal{T}_0+\delta\mathcal{T}$ with $\delta \mathcal{T}$ is the small variation of temperature in superfluid helium and $q^{-1}$ is the second sound velocity $S^2 \rho_s \mathcal{T}_0/(\rho_n c_v)$ with the specific heat capacity $c_v$ at constant volume.  
We refer to \cite{Khala, RT, Andreas} for  detailed discussions on the two fluid model and second sound wave, and chapter $8$ in \cite{Khala} for a discussion of  relevant boundary conditions.


When a quench occurs on a cavity surface, heat does not propagate by diffusion as it does in classical fluids; instead, temperature in the superfluid helium propagates by a wave. 
In the second sound system, by measuring the time delays between the quench and arrival of the second sound wave to detectors, in principle, the distance from the detector to the quench location is determined by the multiplication of theoretical second sound velocity and time delays. There are two distinct detectors in the second sound system to measure second sound waves that are high-sensitivity resistive temperature detectors \cite{SSBB, SSMBB} and oscillating superleak transducers (OSTs) \cite{SE}. The former uses thermometers to measure the thermal fluctuations of the second sound wave, while OSTs detect the fluctuation of counterflow velocity of the wave. 
In 2008, these special temperature microphones, OSTs, were first introduced by Conway, Hartill, Padamsee, and Smith at Cornell University \cite{Conway08, Conway} to detect second sound emitted by a quench. The technique depends on the high sensitivity of OSTs to the fluctuation of superfluid and normal fluid counterflow. By observing the first high fluctuation, the time of arrival of waves can be detected. 




Although the second sound system for SRF quench detection has been successful, there has been interest in the physics community to better understand two issues: uncertainties of the the exact location of the quench and the underlying physical process during emission of the entropy waves.
Many of  second sound systems assume that the quench is a singular point on the cavity surface, and then triangulate the exact quench location from the response of a few OSTs \cite{Conway, EM, MS}. That is to find the intersection point of spheres with centers at OSTs and radii equal to distances to the quench which are determined by travel time of second sound wave and its theoretical velocity. However, in practice, they have experienced  difficulties in exactly locating a quench. 
For instance, when the real wave velocity is higher than the theoretical one, the radii are too small to have an intersection point on the cavity surface, see \cite{EM15, Liu, MS, Plouin} and the references therein. 
The size of the quench is another potential source of inaccuracy. On the outer surface of the cavity, the size of heat source depends on dynamics of the heat propagation on the surface and the superfluid helium. Therefore, signals arriving at different OSTs might come from different positions of the quench spot \cite{Liu}. This results in an overestimated wave velocity for a closer OST; and therefore, the usage of the measured travel time in computing radii in triangulation might not provide precise determination of the quench position. This issue regarding a large quench spot has been studied in laboratories; we refer to papers \cite{Liu, MS}. More sophisticated models of heating at  the boundary of the SRF have also been studied in \cite{EM15}.



In this paper, we mathematically investigate the detection of a quench on the cavity surface by using second second wave that emitted on the quench propagates in superfluid helium. 
As mentioned above, the arrival second sound wave can be detected by special thermometers that are very sensitive to small change of temperature and also detectable by OSTs which are widely employed in many laboratories, however there are many factors that can affect the accuracy of travel time. 
Further investigations are needed for a better understanding of physical process involved during emission and propagation of the second sound wave in order to precisely localize a quench.
To achieve this, we place thermometers around the OST such that the temperature information on the boundary of OST can be detected. More precisely, we apply both temperature data detected by thermometers and an approximate travel time obtained by an OST to derive a mathematical formula for the distance between a quench location and an OST. This formula provides an exact distance information that can be utilized to recover the position and geometry information of the quench.
We would like to note that cavities equipped with OSTs and thermometry sensors have been considered in physical papers \cite{Bertucci, TE} in which this hybrid technique combines with two methods, OSTs and T-mapping, in order to cross check the consistency of their results.





The mathematical model and the main result are introduced as follows. 
Let $D$ represent the cavity in $\R^3$ and be an open bounded and connected subset of $\R^3$ with smooth boundary $\p D$.
Here we suppose the cavity's location and shape are known. 
Assume that one quench occurs on part of the cavity surface and the second sound wave emitted from the quench propagates in the helium bath to an OST.  

We denote the variation of temperature by $u = \delta \mathcal{T}$, which satisfies the wave equation with associated boundary conditions. 
Let $T_0$ be the observed time of the second sound wave propagation.
Then the mathematic model for the second sound wave in a superfluid helium bath is modeled as follows:  
\begin{align}\label{wave}
     \left\{  \begin{array}{rll}
         \partial^2_t u- \Delta u &=\ 0 & \hbox{in } (0, T_0)\times (\R^3\setminus \overline{D})  ,\\
         u|_{t=0} &  =\ 0, &\\ 
         \partial_t u|_{t=0} &=\ 0,\\
         u|_{(0,T_0) \times \p D } &=\ f,\\ 
         \p_\nu u|_{(0,T_0) \times\p D } &=\ g, 
     \end{array} 
    \right.
\end{align}
where $f\in L^2(0,T_0;H^1(\p D))$ and $g\in L^2(0,T_0;L^2(\p D))$ with the support of $f$ and $g$ in spatial dimensions contained in $\Gamma$. Here the quench spot is represented by $\Gamma$ which is an open and connected smooth surface of the boundary $\p D$, and $\nu$ is the unit outer normal to $D$. Notice that since the adjacent material around a small defect is still superconducting, there is no considerable heat loss there. Thus one can consider the boundary conditions $f$ and $g$ are compactly supported on a subset of the cavity surface $\p D$. 
Suppose that $f\geq \mu >0$ and $-g\geq \mu> 0$ for some $\mu>0$
and $\|f\|_{L^2(0,T_0;H^1(\p D))}, \|g\|_{L^2(0,T_0;H^1(\p D))}\leq M$ for some constant $M>0$.   An analysis can also be performed on \eqref{wave} under Robin-type boundary conditions under similar support and lower bound assumptions on the boundary data.
\begin{figure}[htbp] 
         \centering
         \begin{minipage}[b]{.25\textwidth} 
             \includegraphics[width=1.3in]{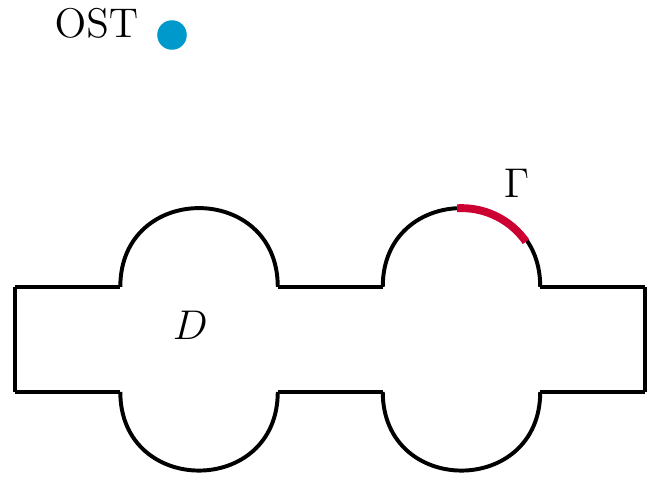}   
         \end{minipage}  
         \caption{A simple layout of quench detection: One quench occurs on the surface of 2-cell cavity with an OST installed in the superfluid helium. }  
\end{figure}

We denote an OST by an open ball $B=B(p,r)$ centered at $p$ with radius $r$ such that $B\subset \R^3\setminus \overline{D}$.
Let $\chi_B$ be a characteristic function on $B$. Then, for fixed constant $\tau>0$, there exists a solution $v\in H^1(\R^3)$ to the equation
$$
    \Delta v-\tau^2 v+\chi_B =0\ \ \hbox{in }\R^3.
$$
The collected dynamical data over a finite time interval on $\p B$ 
is described by the indicator function $I_{\p B}$, defined as follows: 
$$
    I_{\p B}(\tau)=\int_{\p B}\int^{T_0}_0 e^{-\tau s} (u(s,x) \p_\nu v - v\p_\nu u(s,x)) dsdS,
$$
where $\nu$ is the unit outer normal to $B$. 

Suppose $D$ and position of each OST are known. 
Notice that the boundary data $u$ on $(0,T_0)\times \p B$ determines the normal derivative of $u$ on $(0,T_0)\times \p B$, see proposition \ref{neumann}. 
The inverse problem is to find the quench location, $\Gamma$, from the information $u$ on $(0,T_0)\times \p B$ where $T_0$ can be determined by the OST.



We define the shortest path between $x,y\in \R^3\setminus D$ as follows. Let $d(x,y)$ denote the infimum of the lengths of all the piecewise $C^1$ paths in $\R^3\setminus D$ joining $x$ to $y$. It is clear that when $D$ is an empty set, we have $d(x,y)=|x-y|$.  
We define the distance between two subsets $A_1$ and $A_2$ in $\R^3\setminus D$ by 
$$d(A_1, A_2)=\inf\{d(x,y):\ x\in A_1,\ y\in A_2\}.$$ We also define the Euclidean distance of $A'_1$ and $A'_2$ in $\R^3$ by $$d_e(A'_1, A'_2)=\inf\{|x-y|:\ x\in A'_1,\ y\in A'_2 \}.$$ Since $|x-y|\leq d(x,y)$ for all $x,y\in\R^3\setminus D$, it implies that $d_e(A_1, A_2)\leq d(A_1, A_2)$ for subsets $A_1$ and $A_2$ in $\R^3\setminus D$.
 
A quench can be detected if the following conditions are satisfied. Suppose that for any $x\in \Gamma$, $x$ satisfies $-\alpha <\nu(x)\cdot(y-x)<\beta$ for all $y\in B$, where $\alpha,\beta>0$.
Suppose that the subset 
$$\Gamma_0(0)=\{x\in \Gamma:\ \nu(x)\cdot (y-x)>0\ \hbox{for all }y\in B\}$$
of $\Gamma$ has nonzero measure. Moreover, suppose that for all $x\in \Gamma_1=\Gamma\setminus\Gamma_0(0)=\{x\in \Gamma:\ \nu(x)\cdot (y-x)\leq 0\ \hbox{for some }y\in B\}$, $x$ satisfies $d_e(x,B)>d(\Gamma, B)$.
Suppose that the shortest distance between the quench, $\Gamma$, and an OST, $B$, is a direct light of sight path which implies that $d_e(\Gamma,B)=d(\Gamma, B)$. Here we note that this assumption is reasonable since the second sound travels in straight line to OSTs from a quench is a general assumption in papers in physics.




Suppose $d_e(x,p)$ has either a nondegenerate minimum on $\overline\Gamma$ or a minimum on boundary of $\Gamma$ which is not a critical point.
Then the distance from a quench spot to an OST is characterized by the following formula.

\begin{theorem}\label{main}
If $T_0>d(\Gamma,B)$, then
\begin{align}\label{main_id}
    \lim_{\tau\rightarrow\infty} \tau^{-1}\log I_{\p B}(\tau) = -d(\Gamma,B).
\end{align}
Moreover, if $\Gamma=\emptyset$ (empty set), then for all $T_0>0$, we have
\begin{align}\label{main_id1}
    \lim_{\tau\rightarrow\infty} e^{\tau T_0} I_{\p B}(\tau) =0;
\end{align}
if $\Gamma\neq\emptyset$, then for all $T_0>d(\Gamma, B)$, we have
\begin{align}\label{main_id2}
    \lim_{\tau\rightarrow\infty} e^{\tau T_0} I_{\p B}(\tau) =\infty.
\end{align}
\end{theorem}

%

Quench localization  
is different from inverse obstacle scattering problems in which  scattering waves are used to detect hidden obstacles in a medium. In particular, inverse scattering takes advantage of a  wave source outside a known surface that surrounds the obstacle, 
and then one infers information on the obstacle from  the scattering data collected over the surface in a finite time period. 
The enclosure method, first introduced by Ikehata in \cite{I2000}, 
was developed for inverse obstacle scattering problems and gave the distance between the source and the unknown obstacle from the asymptotic behavior of the indicator function. 
Though our setting is very different from  inverse obstacle scattering, we are inspired by work on the enclosure method in \cite{I10, Ireview, I2, I3, I15, I4}, and we  develop quantitative estimates on this indicator function in order to establish a distance formula from the OST to the quench.


We remark several features of the result in this paper.
First, we do not need the whole physical process and boundary information on the interface between the superfluid helium and the cavity. 
Partial information on the boundary data is sufficient in the derivation of the distant formula \eqref{main_id}. 
This is also coincide with the experimental setting in which only when resolving power is small, the thermal transport on the niobium surface needs to be considered.  
Second, the derivation of the formula only depends on sufficiently large observation time that does not need to be exact. 
This is different from the experiments by using OSTs solely in which the time of propagation of the second sound waves is important since timing uncertainties could only give an approximate location  of the quench \cite{Liu, MS}. Moreover, the size of a quench in our setting is not necessarily to be a singular point while it is a common assumption in papers in physics. In addition, we also allow part of the quench area can be outside of the OST's line of sight.



This paper is organized as follows. In section 2, we show that the quench can be detected by a distance formula that gives the distance from the quench to the OST. A further discussion on this problem is stated in section 3.

\section{Detection of a quench on the cavity surface}
In this section, we will derive a formula which reveals the shortest distance between the location of a quench spot on the surface of the cavity and an OST. 

We start by showing that knowing $u$ on $(0,T_0)\times \p B$, one can recover the Neumann data of $u$ on $(0,T_0)\times \p B$ as well. 
\begin{proposition}\label{neumann}
Let $u$ solve (\ref{wave}). 
Then $u|_{(0,T_0)\times \p B }$ determines uniquely $u$ in $(0,T_0)\times \overline B$ and the normal derivative of $u$ on $(0,T_0)\times\p B $.
\end{proposition}
\begin{proof}
    Let $\Phi$ satisfy the wave equation
    \begin{align*} 
    \left\{ \begin{array}{rll}
    \p^2_t \Phi-\Delta \Phi &=\ 0  &\hbox{in } (0,T_0)\times B ,\\
    \Phi|_{t=0}   &=\ 0, \\
     \p_t\Phi|_{t=0} &=\ 0, & \\
     \Phi|_{(0,T_0)\times \p B}& =\ u|_{(0,T_0)\times \p B}.  \\
    \end{array} \right.
    \end{align*}
    Then $u-\Phi$ satisfies the wave equation $(\p_t^2-\Delta)(u-\Phi)=0$ with zero Dirichlet data and zero initial data. Then $u=\Phi$ in $(0,T_0)\times B$ and $\p_\nu u=\p_\nu \Phi$ in $(0,T_0)\times \p B$.
\end{proof}

For a fixed $\tau>0$, the equation
\begin{align}\label{equ_v}
    \Delta v-\tau^2 v+\chi_B =0\ \ \hbox{in }\R^3 
\end{align}
has a $C^2$ solution of the form
$$
    v(x)= {1\over 4\pi} \int_B {e^{-\tau |x-y|}\over |x-y|} \chi_B(y) dy,\ \ x\in\R^3,
$$
known as the volume potential, we refer to \cite{KH}.

By the assumed regularity of $f\in L^2(0,T_0;H^1(\p D))$ and $g\in L^2(0,T_0; L^2(\p D))$, from \cite{Lion}, there exists a unique solution $u$ such that $u\in L^\infty(0,T_0; H^1(\R^3\setminus \overline{D}))$ and $\p_t u\in  L^\infty(0,T_0; L^2(\R^3\setminus \overline{D}))$. 
Similar to \cite{I10}, we denote $W_\tau(x) =\int^{T_0}_0 e^{-\tau s} u(s,x) ds.$ Then $W_\tau$ is the weak solution to the equation
$$
    \Delta W_\tau -\tau^2 W_\tau = e^{-\tau T_0}(\tau u(T_0,x) + \p_t u(T_0,x))  \ \ \hbox{in }\R^3\setminus\overline{D} 
$$
with the Dirichlet boundary 
$
   f_\tau(x)= \int^{T_0}_0 e^{-\tau s}f(s,x) ds 
$
and the Neumann boundary
$
   g_\tau(x)= \int^{T_0}_0 e^{-\tau s}g(s,x) ds. 
$

Applying integration by parts, the integral on the boundary of $B$ is transfered to the one on the quench $\Gamma$ and a lower order term. Since $f$ and $g$ are supported in $\Gamma$, we have
\begin{align}\label{id_z}
     I_{\p B}&=\int_{\p B} (W_\tau \p_\nu v - v\p_\nu W_\tau) dS\notag\\
     &= - \int_{\p D} (W_\tau \p_\nu v - v\p_\nu W_\tau) dS + \int_{(\overline{B\cup D})^c} e^{-\tau T_0}(\tau u( T_0,x)+\p_t u(T_0, x)) v dx \notag \\
     &= \int_{\Gamma} (f_\tau  \p_\nu v - v g_\tau ) dS + \int_{(\overline{B\cup D})^c} e^{-\tau T_0}(\tau u(T_0,x)+\p_t u(T_0,x)) v dx.
\end{align}

We begin by showing the lower bound for the indicator function. We estimate the first integral in \eqref{id_z} in the Lemma 2.2-2.4. To derive the lower bounded of the first integral, since $f$ and $g$ are bounded below by a positive constant $\mu$, we only focus on the estimate of $\int v dS$ and $\int \p_\nu vdS$. Thus, it is crucial to quantitatively study the function $v$. 
The proof of lemma 2.2 can be found  in \cite{I10} and is provided below for the sake of completeness. 
\begin{lemma}
  There exists $\tau_0>0$ such that for all $\tau\geq \tau_0$, we have
  \begin{align}\label{est_v}
      \tau^2 e^{\tau d_e(x,B)}v(x)\geq C_0>0,
  \end{align}
  where $C_0$ is independent of $x\in\Gamma$.
\end{lemma}
\begin{proof}  
Fix $x\in \Gamma$, we consider the spherical coordinate. Then any point $y$ in $B(p,r)$ can be expressed as
$
    y= x +\rho \omega 
$
with
$$
    R_-(\omega) < \rho < R_+(\omega),\ \omega\in E_x,
$$
where $E_x=\{\omega\in S^2:\ (p-x)\cdot \omega > \sqrt{|p-x|^2-r^2}\}$ and $$
    R_\pm(\omega)=(p-x)\cdot\omega\pm\sqrt{r^2-|p-x|^2+((p-x)\cdot\omega)^2}.
$$  
Hence by changing of coordinates we have
\begin{align*}
 v(x)&= {1\over 4\pi} \int_B {e^{-\tau |x-y|}\over |x-y|} \chi_B(y) dy \\
&= {1\over 4\pi} \int_{E_x}\int_{R_-(\omega)}^{R_+(\omega)} e^{-\tau \rho}  \rho  d\rho d\omega \\
&\geq -{d_e(x,B)\over 4\pi\tau} \int_{E_x}  e^{-\tau \rho}|_{R_-(\omega)}^{R_+(\omega)} d\omega\\
&=   {d_e(x,B)\over 4\pi\tau} \int_{E_x}  \LC e^{-\tau R_-(\omega)}-e^{-\tau R_+(\omega)}\RC d\omega .
\end{align*}
Notice that $d_e(x,B)=|p-x|-r$. From the definition of $R_+(\omega)$, it implies that $R_+(\omega)>\sqrt{|p-x|^2-r^2} =\sqrt{d_e(x,B) }\sqrt{|p-x|+r}$, which leads to
\begin{align}\label{R_p}
    \int_{E_x}e^{-\tau R_+(\omega)}d\omega&\leq m(E_x)e^{-\tau \sqrt{d_e(x,B)} \sqrt{|p-x|+r}} \notag\\
    &=m(E_x)e^{-\tau d_e(x,B)} e^{ \tau(d_e(x,B)- \sqrt{d_e(x,B)} \sqrt{|p-x|+r})}.
\end{align}
Here we denote the measure of set $E_x$ by $m(E_x)$.
In addition, we have $$d_e(x,B)- \sqrt{d_e(x,B)} \sqrt{|p-x|+r}<0.$$ Thus the integral	 
\begin{align}\label{R_+}
    e^{\tau d_e(x,B)}\int_{E_x}e^{-\tau R_+(\omega)}d\omega \leq  m(E_x) e^{ \tau(d_e(x,B)- \sqrt{d_e(x,B)} \sqrt{|p-x|+r})}
\end{align}
is decreasing exponentially as $\tau$ increases.	

Now we will estimate the integral $\int_{E_x}e^{-\tau R_-(\omega)}d\omega.$ First, for $\omega\in E_x$, we write
\begin{align*} 
    \omega(\phi,\theta) = \LC\sin\phi\cos\theta, \sin\phi\sin\theta, \cos\phi\RC,
\end{align*} 
where $0\leq \theta\leq 2\pi$, $0\leq \phi\leq \sin^{-1}(r/|p-x|)$.
Then $R_-(\omega)$ can be expressed as
\begin{align*} 
    R_-(\omega(\phi,\theta))  &= |p-x|\cos\phi - \sqrt{r^2-|p-x|^2\sin^2\phi}.
\end{align*}
Furthermore 
we need the following inequality: 
\begin{align*} 
    R_-(\omega(\phi,\theta))-d_e(x,B)
    &\leq |p-x|\cos\phi-\sqrt{r^2-|p-x|^2\sin^2\phi}-|p-x|+r\\
    &=|p-x|\LC\cos\phi -1 \RC+\LC r-\sqrt{r^2-|p-x|^2\sin^2\phi} \RC\\
    &= |p-x|\LC{ -\sin^2\phi \over  \cos\phi +1} + {|p-x|\sin^2\phi  \over r+\sqrt{r^2-|p-x|^2\sin^2\phi}}\RC\\
    &\leq |p-x|{|p-x|\sin^2\phi  \over r+\sqrt{r^2-|p-x|^2\sin^2\phi}}\\
    &\leq {|p-x|^2 \sin^2\phi \over r}.
\end{align*}
The integral can be written as
\begin{align}\label{R_-}
    \int_{E_x}e^{-\tau R_-(\omega)}d\omega  
    &=\int_{0}^{2\pi} \int_0^{\sin^{-1}(r/|p-x|)}  e^{-\tau R_-(\omega(\phi,\theta))} \sin\phi d\phi d\theta \notag \\
    &=e^{-\tau d_e(x,B)}\int^{2\pi}_0  \int_0^{\sin^{-1}(r/|p-x|)}  e^{-\tau (R_-(\omega(\phi,\theta)) -d_e(x,B))} \sin\phi d\phi d\theta \notag \\
    &\geq e^{-\tau d_e(x,B)}\int^{2\pi}_0  \int_0^{\sin^{-1}(r/|p-x|)} e^{-\tau |p-x|^2\sin^2\phi /r} \sin\phi d\phi d\theta \notag\\
    &\geq e^{-\tau d_e(x,B)}\int^{2\pi}_0  \int_0^{\sin^{-1}(r/|p-x|)} e^{-\tau |p-x|^2\sin^2\phi /r} \sin\phi \cos\phi d\phi d\theta \notag\\
    &=  { e^{-\tau d_e(x,B)}\pi r\over  \tau |p-x|^2 }  \LC 1- e^{-\tau r}  \RC .
 \end{align}  
From \eqref{R_+} and \eqref{R_-}, we deduce that for $x\in\Gamma$, the following estimate holds:
\begin{align*}
        &\tau^2 e^{\tau d_e(x,B)}v(x)\\
        &\geq {d_e(x,B)  r\over 4 |p-x|^2}(1-e^{-\tau r})- {d_e(x,B)\tau \over 4\pi} m(E_x)e^{ \tau(d_e(x,B)- \sqrt{d_e(x,B)} \sqrt{|p-x|+r})}.
\end{align*}
When $\tau$ is sufficiently large, the above estimate is controlled by ${d_e(x,B)r\over 4 |p-x|^2}$ that is bounded from below by a positive constant which can be chosen to be independent of $x$ in $\Gamma$.
This completes the proof.
\end{proof}

Suppose there exists a diffeomorphism $\phi$ of $\Gamma$ such that $\phi:R\rightarrow \Gamma$ whose domain is an open subset $R$ in $\R^2$. Let $h(s,t)=|\phi(s,t)-p|$ for all $(s,t)\in R$. Let $(s_0,t_0)$ be the global minimum of $h$ on $\overline{R}$. 
We need the following estimate to derive the lower bound for \eqref{id_z}.

\begin{lemma}\label{laplace}
There exists $\tau_0>0$ such that for all $\tau\geq \tau_0$,
if $(s_0,t_0)$ is a nondegenerate critical point of $h$, then 
   \begin{align}\label{lower_1}
    \tau^{3	}e^{\tau d(\Gamma, B)}\int_\Gamma v(x)dS >0;
   \end{align} 
if $(s_0,t_0)$ is on the boundary of $R$ and is not a critical point of $h$, then 
\begin{align}\label{lower_2}
 \tau^{7/2	}e^{\tau d(\Gamma, B)}\int_\Gamma v(x)dS >0. 
\end{align} 
\end{lemma}
\begin{proof}
From estimate \eqref{est_v}, for $\tau\geq \tau_0$, we have
   \begin{align*}
         e^{\tau d(\Gamma, B)}\int_\Gamma v(x)dS
        &\geq 
         \tau^{-2}e^{\tau d(\Gamma, B)}\int_\Gamma   e^{-\tau d_e(x,B)}dS  \\
         &\geq 
         \tau^{-2}e^{\tau d(p,\Gamma)}\int_\Gamma    e^{-\tau |x-p|}dS.
   \end{align*}
In the following we will show that $\int_{\Gamma}e^{-\tau |x-p|}dS$ is bounded away from zero when $\tau$ is sufficiently large. By the definition of surface integral, we have
$$
    \int_\Gamma    e^{-\tau |x-p|}dS = \int\int_R e^{-\tau |\phi(s,t)-p|}  |\phi_s\times\phi_t |dsdt=:J(\tau).
$$
Since $p$ is not in $\Gamma$, it implies that $h(s,t)$ is a smooth function on $R$. The global minimum of $h$ could be either in the interior of $R$ or on the boundary of $R$. 
We apply the Laplace method to derive a nonzero lower bound for $J(\tau)$, then the two estimates in this lemma hold.
In the following, we first consider the minimum occurs inside $R$. Then there exists a point $(s_0,t_0)$ in the interior of $R$ such that $h(s_0,t_0)=d(p,\Gamma)$. Since $(s_0,t_0)$ is an interior point, we can find a small open ball $R_0$ containing $(s_0,t_0)$ and $R_0$ is in the interior of $R$. Moreover, we choose $R_0$ such that the Morse's lemma applies. Then there are neighborhood $D_0$ of $0$ in $\mathbb{R}^3$ and a diffeomorphism $\psi: D_0\rightarrow R_0$ so that $\psi(0)=(s_0,t_0)$. By changing of variables, $\psi(\zeta)=(s,t)$, we have  
\begin{align*}
   J(\tau)&\geq \int\int_{R_0} e^{-\tau h(s,t)}|\phi_s\times\phi_t|dsdt\\
   &=e^{-\tau h(s_0,t_0)} \int\int_{\psi^{-1}(R_0)} e^{-{1\over 2}\tau (\mu_1\zeta_1^2+\mu_2\zeta_2^2)} \tilde{\phi}(\zeta)  d\zeta,
\end{align*}
where $\zeta=(\zeta_1,\zeta_2)$ and $\mu_1,\mu_2$ are positive eigenvalues of Hessian of $h$ at $(s_0,t_0)$. Here we denote $\tilde \phi(\zeta)=|\phi_s\times\phi_t|(\psi(\zeta))\text{det}{(D \psi (\zeta))}$. By Taylor's theorem, we have
$$
    \tilde \phi(\zeta)=\tilde \phi(0) + \nabla \tilde{\phi}(\xi )\cdot \zeta
$$
for $\xi \in D_0$.
Integrating each term, $J(\tau)$ is bounded from below by the leading term:
\begin{align*}
    J(\tau)\geq  e^{-\tau h(s_0,t_0)} \tilde{\phi}(0) \tau^{-1}
\end{align*}
for $\tau$ large and some constant $C$ depending on $\mu_1,\mu_2$. It leads to the lower bound \eqref{lower_1}.

Now we consider the case where
$h$ has a minimum at the boundary point $(s_0,t_0)$. Then $h(s_0,t_0)=d(p,\Gamma)$. If $\nabla h(s_0,t_0)=0$, then by following a similar argument as above, we have the same order as in the previous case. Thus, we only discuss the situation when $(s_0,t_0)$ satisfies $\nabla h(s_0,t_0)\neq 0$. Suppose there exists an open subset $\tilde R$ of $R$ such that $(s_0,t_0)\in \tilde R$ and $\nabla h\neq 0$ in $\tilde R$. 
We denote $k(s,t)= |\phi_s\times\phi_t|$ and define functions $n_0,n_1,$ and $k_1$ by 
$$
    n_0 ={\nabla h\over|\nabla h|^2}  k,\ \  n_1={\nabla h\over|\nabla h|^2}  k_1,\ \ k_1 =\nabla\cdot  n_0. 
$$ By the divergence theorem, we have
\begin{align}\label{Jtau}
   J(\tau) 
   &\geq \int\int_{\tilde R} e^{-\tau h(s,t)} k(s,t)dsdt \notag\\
   &={1\over \tau^2} \int\int_{\tilde R} (\nabla\cdot n_1) e^{-\tau h }dsdt-\sum^1_{j=0}{1\over\tau^{j+1}}\int_{\p \tilde R } (n_{j}\cdot \nu )e^{-\tau h }dl,
\end{align}
where $dl$ is the line integral, $\nu$ is the normal to the boundary of $\tilde R $ (that is, $\p \tilde R $).
The leading term is 
\begin{align*} 
    -{1\over\tau}\int_{\p\tilde R } (n_{0}\cdot \nu)e^{-\tau h }dl.
\end{align*}
Assume that $\p \tilde R $ is parametrized by $\gamma(z)$, for $z$ in the parameter domain $I$ so that $(s_0,t_0)=\gamma(z_0)$ for some interior point $z_0$ in $I$. By the definition of the line integral, we have
\begin{align*} 
\int_{\p \tilde R } (n_{0}\cdot \nu)e^{-\tau h }dl=\int_I (n_0\cdot \nu)(\gamma(z)) e^{-\tau h( \gamma(z)) } |\gamma'(z) | dz.
\end{align*}
Let $\tilde h(z)=h(\gamma(z))$. Since $z_0$ is the interior point of $I$ and the minimizer of $\tilde h(z)$, it implies that $\tilde h'(z_0)=0$ and $\tilde h''(z_0)> 0$.
Since $\tilde h$ is smooth, for $\varepsilon>0$, there exists  $\delta >0$ such that if $|c-z_0|<\delta$, then $0< \tilde h''(c)<\tilde h''( z_0 )+\varepsilon$. By Taylor's theorem, for any $|z-z_0|<\delta$, we have
$$
   -\tilde  h(z)\geq -\tilde h(z_0)-{1\over 2} ( \tilde h''(z_0)+\varepsilon)(z-z_0)^2.
$$
Note that since $z_0$ is a minimum, $\tilde h'(z_0)=0$ implies that $\nu(\gamma(z_0))=-\nabla h/|\nabla h|(s_0,t_0)$. Then $-(n_0\cdot \nu)(\gamma(z_0))= k(\gamma(z_0)) /|\nabla h|$ is positive at $z_0$. By choosing $\delta$ sufficiently small such that $-(n_0\cdot \nu)> 0$ in $(z_0-\delta,z_0+\delta)$, then we have the following lower bound:
\begin{align*}
    &\int_{\p \tilde R } -(n_{0}\cdot \nu)e^{-\tau h }dl\\
    &\geq e^{-\tau \tilde h(z_0)}  \int_{z_0-\delta}^{z_0+\delta} -(n_0\cdot \nu)(\gamma(z)) e^{ - {1\over 2}\tau  ( \tilde h''(z_0)+\varepsilon)(z-z_0)^2 } |\gamma'(z) |  dz-O(e^{-\tau(c+\tilde h(z_0))})\\
    &=\eta(\tau)^{-1} e^{-\tau \tilde h(z_0)}  \int^{ \delta\eta(\tau)}_{ -\delta \eta(\tau)}-(n_0\cdot \nu)(\gamma(z(y))) e^{ - y^2 } |\gamma'(z(y)) |  dy-O(e^{-\tau(c+\tilde h(z_0))}),
\end{align*}
where $\eta(\tau) = \sqrt{ {1\over 2}\tau  ( \tilde h''(z_0)+\varepsilon)}$ and constant $c$ is independent of $\tau$. Here the identity comes from the change of variables, $$y=\eta(\tau)(z-z_0).$$
We denote $ q(y) = -(n_j\cdot \nu)(\gamma(z(y))) |\gamma'(z(y))|  $. 
Applying Taylor's theorem again, there exist an interval $(-\beta,\beta)$ with $0<\beta<\delta\eta(\tau)$ and a point $c$ in it, then we have
$$
   q(y)=q(0)+q'(c)y
$$
for all $y$ in the interval $(-\beta,\beta)$. Then we have
\begin{align*} 
&\int^{ \delta\eta(\tau)}_{ -\delta\eta(\tau)}-(n_0\cdot \nu)(\gamma(z(y))) e^{ - y^2 } |\gamma'(z(y)) |  dy \\
&\geq \int^{\beta}_{-\beta} q(0) e^{-y^2} dy+  \int^{\beta}_{-\beta} q'(c)y e^{-y^2} dy\geq 2 q(0)\beta e^{-\beta^2}.
\end{align*}
From this inequality, we deduce that 
\begin{align}\label{leading}
   \int_{\p \tilde R }- (n_{0}\cdot \nu)e^{-\tau h }dl
    \geq  2  q(0) \beta e^{-\beta^2}\eta(\tau)^{-1} e^{-\tau \tilde h(z_0)} .
\end{align}
We recall that $q(0)$ can be expressed by the original function $k$, that is, $$q(0)={k(s_0,t_0) |\gamma'(z_0) |\over |\nabla h(s_0,t_0)|}.$$
Combining \eqref{Jtau} and \eqref{leading}, if $(s_0,t_0)$ is a minimum and $\nabla h(s_0,t_0)\neq 0$, then we have 
$$J(\tau) \geq O(\tau^{-3/2}e^{-\tau h(s_0,t_0)})$$ as $\tau$ goes to infinity. This leads to the inequality \eqref{lower_2}.

\end{proof}

\begin{lemma} Let $\delta=3$ if $h$ has a nondegenerate minimum; otherwise, $\delta=7/2$ if $h$ has a minimum on $\p R$ and it is not a critical point. Then the estimate
    \begin{align}\label{inf_1}
     \liminf_{\tau \rightarrow \infty} \tau^{\delta +1} e^{\tau d( \Gamma , B)}\int_{\Gamma}  f_\tau \p_\nu v-vg_\tau dS > 0 
    \end{align}  
    holds provided that $T_0 >  d(\Gamma, B)$. 
    Moreover, one has the following estimate
    \begin{align}\label{liminf}
    \liminf_{\tau\rightarrow \infty} \tau^{\delta+1} e^{\tau d( \Gamma  ,B)} I_{\p B} (\tau)>0.
    \end{align}
\end{lemma} 

\begin{proof}

We now estimate $\int_{\Gamma} \p_\nu v(x)dS$.
Recall that for any $x\in \Gamma$(connected), $ -\alpha<\nu(x)\cdot(y-x)< \beta $ for all $y\in B$ where $\alpha, \beta>0$.
We denote for $\varepsilon>0$,
$$
    \Gamma_0(\varepsilon)=\{x\in \Gamma:\ \nu(x)\cdot(y-x)>\varepsilon\ \hbox{for all }y\in B\}
$$
and $\Gamma_1=\Gamma\setminus\Gamma_0(0)$.  
There exists a point $x_0$ on $\overline\Gamma$ such that
$
    d_e(x,B)\leq d_e(x_0, B)
$
for all $x\in \overline\Gamma$. We deduce that
\begin{align*}
 \int_{\Gamma} \p_\nu v(x)dS 
 &\geq\int_{\Gamma_0(0)}\p_\nu v(x) dS  -\LV  \int_{\Gamma_1}\p_\nu v(x)dS\RV \\ 
 &\geq   \int_{\Gamma_0(\varepsilon)} {1\over 4\pi}\int_B \varepsilon\LC\tau + {1\over |x-y|}\RC  {e^{-\tau |x-y|}\over |x-y|^2} dydS - L\\
  &\geq \int_{\Gamma_0(\varepsilon)}{\varepsilon \over (d(x_0,B)+2r) } \LC\tau + {1\over d(x_0,B)+2r}\RC  v(x)dS - L,
\end{align*}
where $L$ is defined in the following estimate
\begin{align*}
 &\LV  \int_{\Gamma_1}\p_\nu v(x)dS\RV \\ 
 &\leq  (\alpha \vee \beta) \int_{\Gamma_1}{1\over 4\pi}\int_B \LC\tau + {1\over |x-y|}\RC  {e^{-\tau |x-y|}\over |x-y|^2} dydS \\
  &\leq (\alpha \vee \beta)\int_{\Gamma_1 } {1\over 4\pi}m(B)\LC\tau + {1\over d_e(x, B)}\RC{e^{-\tau d_e(x, B)} \over d_e(x, B)^2} dS=:L.
\end{align*}
Here $\alpha\vee\beta=\text{max}\{\alpha, \beta\}$.
Notice that the result in lemma \ref{laplace} also works for a small region $\Gamma_0(\varepsilon)$. Therefore from the hypothesis $T_0>d(\Gamma, B)$ and $d_e(x, B)>d(\Gamma,B)$ for all $x\in\Gamma_1$, we obtain
\begin{align}\label{lower_3}
    \liminf_{\tau\rightarrow \infty} \tau^{\delta}e^{\tau d(\Gamma, B)}\int_{\Gamma} \p_\nu v(x)dS >0.
\end{align} 
Since $f$ and $-g$ are bounded below by $\mu>0$, by definitions of $f_\tau$ and $g_\tau$, we have
$$
    f_\tau, -g_\tau \geq \mu{1-e^{-\tau T_0}\over \tau}.
$$
Then by lemma \ref{laplace} and \eqref{lower_3}, we can derive 
\begin{align}\label{id_z0}
       \liminf_{\tau\rightarrow\infty}\tau^{\delta+1} e^{\tau d(\Gamma,B)}\int_{ \Gamma}  f_\tau\p_\nu v-vg_\tau dS   > 0,
\end{align}
which is the first estimate in the lemma.

Now we consider the second integral in \eqref{id_z}.
First we recall that $u\in L^\infty(0, T_0; H^1(\R^3\setminus \overline{D}))$ and $\p_t u\in L^\infty(0, T_0; L^2(\R^3\setminus \overline{D}))$, then one has
\begin{align}\label{u_norm}
     \|\tau u(T_0,x)+\p_t u( T_0,x)\|_{L^2(\R^3\setminus \overline{D} )} = O(\tau)  
\end{align}
as $\tau\rightarrow\infty$.
Second, by integration by parts, from \eqref{equ_v} we have
\begin{align*}
    \int_{\R^3} |\nabla v|^2dx +\int_{\R^3} \tau^2v^2dx=\int_Bvdx.
\end{align*}
Then by H\"older's and Young's inequalities, we obtain the estimate
$$
    \|\nabla v\|^2_{L^2(\R^3)} + \tau^2 \|v\|^2_{L^2(\R^3)}\leq \tau \|v\|^2_{L^2(\R^3)}+{1\over 4\tau}m(B)^2
$$
which leads to
$$\|v\|_{L^2(\R^3)}=O(\tau^{-3/2} ) $$
for $\tau\rightarrow \infty$. 
Combining with \eqref{id_z}, \eqref{id_z0}, $T_0>d(\Gamma, B)$, and using H\"older's inequality again, we derive that 
$$
    \liminf_{\tau\rightarrow \infty} \tau^{\delta+1} e^{\tau d(\Gamma ,B)} I_{\p B} (\tau)>0,
$$
which completes the proof.  
\end{proof}

In order to ensure the existence of the limit in \eqref{main_id}, it remains to estimate the upper bound of the indicator function. 
\begin{proof}[Proof of Theorem \ref{main}.]
We consider the integral which appears in \eqref{id_z}. From the assumptions on $f$ and $g$, we have
\begin{align*}
     | \int_{ \Gamma }  f_\tau\p_\nu v-vg_\tau dS|
     &\leq \|f_\tau\|_{L^2(\Gamma)} \|\p_\nu v\|_{L^2(\Gamma)}+\|g_\tau\|_{L^2(\Gamma)} \| v\|_{L^2(\Gamma)}\\
     &\leq M\LC{1-e^{-2\tau T_0} \over 2\tau} \RC^{1/2} \LC \|\p_\nu v\|_{L^2(\Gamma)}+ \|v\|_{L^2(\Gamma)}\RC.
\end{align*} 
Then we get
\begin{align*}
 v(x) = {1\over 4\pi} \int_B {e^{-\tau |x-y|}\over |x-y|} dy  
&= {1\over 4\pi} \int_{E_x}\int_{R_-}^{R_+} e^{-\tau \rho}  \rho  d\rho d\omega \\
&\leq {d_e(x_0,B)+2r\over 4\pi\tau} \int_{E_x}  \LC e^{-\tau R_-(\omega)}-e^{-\tau R_+(\omega)}\RC d\omega.
\end{align*}
From \eqref{R_p} and the fact that $R_-(\omega)\geq d(\Gamma, B)$, we can deduce that
\begin{align}\label{est2_v}
    |v(x)|\leq  m(E_x) k(\tau){d_e(x_0,B)+2r \over 4\pi \tau},
\end{align}
where 
$$
    k(\tau)=  e^{-\tau d(\Gamma,B)} +e^{-\tau d(x,B)}e^{\tau ( d_e(x, B)-\sqrt{d_e(x, B)}\sqrt{|p-x|+r})} .
$$
Since $d_e(x,B)\geq d(\Gamma, B)$ by hypothesis, it implies that 
\begin{align*}
    \|v\|_{L^2(\Gamma)} \leq m(S^2) m(\Gamma)^{1/2} \tilde k(\tau){d_e(x_0,B)+2r \over 4\pi \tau},
\end{align*}
where constant $C$ 
$$
    \tilde k(\tau)=  e^{-\tau d(\Gamma,B)} +e^{-\tau d(\Gamma,B)}e^{\tau {-2r\sqrt{d_e(x_0, B)}\over \sqrt{d_e(x_0,B)}+\sqrt{d_e(x_0,B)+2r}}  } .
$$
Furthermore, from the estimate
\begin{align*}
    |\p_\nu v(x)|\leq {\alpha\vee\beta\over d(\Gamma,B)}\LC\tau+{1\over d(\Gamma,B)}\RC|v(x)|,
\end{align*}	
we obtain that
\begin{align*}
 \|\p_\nu v\|_{L^2(\Gamma)}\leq {\alpha\vee\beta \over   d(\Gamma,B)}\LC\tau+{1\over d(\Gamma,B)}\RC \|v\|_{L^2(\Gamma)} .
\end{align*}	
Then we have the following estimate
$$
    |\int_{ \Gamma }  f_\tau \p_\nu v-vg_\tau dS|=O(  \tau^{-1/2}e^{-\tau  d(\Gamma,B)}).
$$
From the above estimate, combining with \eqref{id_z}, \eqref{u_norm} and $\|v\|_{L^2(\R^3)}=O(\tau^{-3/2})$, we further deduce that  
\begin{align*}
   \limsup_{\tau\rightarrow\infty} \tau^{1/2} e^{\tau d(\Gamma, B)} I_{\p B}(\tau) <\infty.
\end{align*}
Hence, with \eqref{liminf}, the estimate \eqref{main_id} in Theorem \ref{main} holds. The identities \eqref{main_id1} and \eqref{main_id2} can be derived similarly by following the above argument. Therefore, we complete the proof of Theorem \ref{main}.

\end{proof}

\begin{remark}
From the proof of Theorem \ref{main} above, with one OST, we can derive a lower bound for the size of the quench $\Gamma$, that is, 
\begin{align}\label{size_q}
    \limsup_{\tau\rightarrow \infty} \tau^{1/2} e^{\tau d(\Gamma, B)}  |I_{\p B}(\tau)|\leq  c_0 m(\Gamma)^{1/2} ,
\end{align}
where $c_0={Mm(S^2)\over 4\sqrt{2}\pi}{(\alpha \vee\beta)(d_e(x_0, B)+2r)\over d(\Gamma, B)}$ with $x_0\in \overline{\Gamma}$ satisfies $d_e(x,B)\leq d_e(x_0, B)$
for all $x\in \overline\Gamma$.
 
For example, we consider a simple setting of a quench on a unit disk. By using \eqref{size_q}, we have a lower bound for the radius of this quench spot.
Let the cavity $D=\{(x_1,x_2,0):\ x^2_1+x^2_2\leq 1\}$ and the quench $\Gamma=\{(x_1,x_2,0):\ x_1^2+x_2^2\leq r_q^2\}$ for some scalar $0<r_q<1$. Suppose that $p=(a,b,c)$ with $c>0$ is the center of an OST, $B=B(p,r)$, with $0<r<c$. In spherical coordinates, we write $p=(\rho\cos\theta\sin\phi, \rho\sin\theta\sin\phi, \rho\cos\phi)$ for $\rho>0$, $0\leq \theta\leq 2\pi,\ 0\leq \phi\leq \pi$.
Thus we have  
$$
    d_e(x_0,B)\leq |p+(\cos \theta ,\sin\theta,0)| - r,
$$
which implies that 
$$
    c_0\leq \tilde c_0:={Mm(S^2)\over 4\sqrt{2}\pi}{(\alpha \vee\beta)(|p+(\cos \theta ,\sin\theta,0)| + r)\over d(\Gamma, B)}.
$$
Combining with \eqref{size_q}, we deduce the range for the radius of the quench
\begin{align*}
     1>r_q  \geq \tilde c_0^{-1}\pi^{-1/2} \limsup_{\tau\rightarrow \infty} \tau^{1/2} e^{\tau d(\Gamma, B)}  |I_{\p B}(\tau)|.
\end{align*}
\end{remark}

\section{Conclusion}
In this paper, we rigorously studied quench detection on the surface of a SRF cavity bathed in the superfluid helium. We have established a mathematical formula for the shortest distance between an OST and a quench by using the second sound wave that propagates in the superfluid helium. 
The distance formula derived in this paper can be potentially applied to different models for a quench and cavities with suitable adjustments. For example, the crab cavity has complex geometry that makes a triangulation process is difficult to locate a quench in practice. We would like to note that more information of a quench can be detected if the distance formula is applied on more OSTs.
In particular, to locate a point-like heat source on the cavity in three dimensions which is the widely used model for quench detection in laboratories, the quench point can be detected by utilizing at least three OSTs with their corresponding distances derived by the formula to the quench. For a nonsingular quench spot, its size and shape can be determined by acquiring the distances from a few of OSTs to the quench.

For our further studies on quench detection, it is important and challenging to understand the effect of thermodynamics in cavity surface made of niobium and the transition of heat propagation from the niobium to the superfluid helium. 
In \cite{Liu}, it was observed that  heat could first propagate along the cavity surface and then later excite an entropy wave
in the superfluid helium.  In particular, in  experiments since the speed of propagation is faster in the niobium than in the superfluid helium, the measured travel time at an OST is smaller. Consequently,  the measured quench location is found to be above the cavity surface.
To study quench detection in this  situation with coupled physics, it might be closely related to the coupling of a heat equation with a wave equation, see for example \cite{I17couple, ZuaZua}.
Furthermore, it was observed in laboratories \cite{Liu, MS}, the size of quench could also depend on the dynamics of the heat propagation on the cavity surface. This can  result in an overestimation of the measured wave velocity and hence would  lower the accuracy of quench detection. 







\vskip1cm 
\textbf{Acknowledgment.}
The authors thank Zachary Conway and Daniel Lathrop for stimulating discussions. 
Both authors thank Masaru Ikehata for helpful comments and bringing several references to the authors' attention. 
The second author was supported in part by  NSF grant DMS-1516565.

\bibliographystyle{abbrv}

\bibliography{dipolebib}

\end{document}